\documentclass[a4paper, 12pt]{amsart}
\usepackage{hyperref,upref,amssymb,dsfont}

\numberwithin{equation}{section}

\newtheorem{thm}{Theorem}[section]
\newtheorem{lemma}[thm]{Lemma}

\theoremstyle{definition}
\newtheorem{defn}[thm]{Definition}

\theoremstyle{remark}

\newcommand{\id}{\operatorname{id}}

\newcommand{\dom}{\mathrm{dom}}

\title[Realizing Semigroup C*-Algebras As Groupoid C*-Algebras]{Realizing Semigroup C*-Algebras As Groupoid C*-Algebras}

\date{13, Jun, 2019}

\author[H. Li]{Hui Li}
\address{Hui Li,
Research Center for Operator Algebras and Shanghai Key Laboratory of Pure Mathematics and Mathematical Practice, Department of Mathematics, East China Normal University, 3663 Zhongshan North Road, Putuo District, Shanghai 200062, China}
\email{lihui8605@hotmail.com}

\thanks{The author was supported by National Natural Science Foundation of China (Grant No. 11801176), by Research Center for Operator Algebras of East China Normal University, and by Science and Technology Commission of Shanghai Municipality (Grant No. 18dz2271000).}

\begin{document}

\begin{abstract}
We show that semigroup C*-algebras are groupoid C*-algebras.
\end{abstract}

\maketitle

\section{Introduction}

Xin Li in \cite{Li12} introduced the semigroup C*-algebra which have connections to the number theory. Its definition is as follows.

Throughout the paper, fix a countable unital left cancellative semigroup $P$.

\begin{defn}\label{define semigp C*-alg}
For $p \in P$, we also denote by $p$ the left multiplication map $q \mapsto pq$. Denote the \emph{constructible right ideals} by
\[
\mathcal{J}(P):=\{p_1^{-1}q_1\cdots p_n^{-1}q_nP:n \geq 1, p_1,q_1,\dots, p_n,q_n \in P\} \cup \{\emptyset\}.
\]

Then the \emph{(full) semigroup C*-algebra} $C^*(P)$ of $P$ is the universal unital C*-algebra generated by a family of isometries $\{v_p\}_{p \in P}$ and a family of projections $\{e_X\}_{X \in \mathcal{J}(P)}$ satisfying the following relations.
\begin{enumerate}
\item\label{v_p v_q=v_pq} $v_p v_q=v_{pq}$ for all $p,q \in P$;
\item\label{v_P e_X v_p*}  $v_p e_X v_p^*=e_{pX}$ for all $p \in P,X \in \mathcal{J}(P)$;
\item $e_\emptyset=0_{C^*(P)}$ and $e_P=1_{C^*(P)}$;
\item\label{e_X e_Y} $e_X e_Y=e_{X \cap Y}$ for all $X,Y \in \mathcal{J}(P)$.
\end{enumerate}
Furthermore, define $C^{*(\vee)}(P)$ to be the universal unital C*-algebra generated by a family of isometries $\{w_p\}_{p \in P}$ and a family of projections $\{f_X\}_{X \in \mathcal{J}(P)}$ satisfying the above four conditions and a further condition.
\begin{enumerate}\setcounter{enumi}{4}
\item\label{e_X cup Y} For any $X,Y \in \mathcal{J}(P)$, if $X \cup Y \in \mathcal{J}(P)$, then $f_{X \cup Y}=f_X+f_Y-f_{X \cap Y}$.
\end{enumerate}
\end{defn}

One should distinguish the definition of $C^{*(\vee)}(P)$ and Xin Li's construction of $C^{*(\cup)}(P)$ in \cite[Definition~2.4]{Li12}. The point is that in $C^{*(\vee)}(P)$ in order to perform Condition~(\ref{e_X cup Y}) of Definition~\ref{define semigp C*-alg} we require that not only $X,Y \in \mathcal{J}(P)$ but also $X \cup Y \in \mathcal{J}(P)$. On the other hand, in $C^{*(\cup)}(P)$, the ideals taken into consideration are not only constructible right ideals $\mathcal{J}(P)$ but the finite union of elements of $\mathcal{J}(P)$. It is clear that $C^{*(\vee)}(P)$ is an intermediate C*-algebra between $C^*(P)$ and $C^{*(\cup)}(P)$.

In this short article, we aim to show that semigroup C*-algebras are (possibly non-Hausdorff) groupoid C*-algebras. In order to achieve these we are going to apply the groupoid of germs techniques and we inherit notation from \cite{MR2419901} (one can also refer to \cite{Paterson:Groupoidsinversesemigroups99} for the background of groupoids of germs).

\section{Main Results}

It is well-known that $S:=\{v_{p_1}^*v_{q_1}\cdots v_{p_n}^*v_{q_n}:n \geq 1, p_1,q_1,\dots, p_n,q_n \in P\} \cup \{0_{C^*(P)}\}$ is a unital inverse semigroup with the zero element. Denote by $E:=\{e_X:X \in \mathcal{J}(P)\}$ the set of all idempotents of $S$; denote by $\widehat{E}$ the spectrum of $E$ (notice that $\widehat{E}$ is actually compact because $E$ is unital); denote by $\mathcal{I}(\widehat{E})$ the set of all bijections between subsets of $\widehat{E}$; denote by $\theta:S \to \mathcal{I}(\widehat{E})$ the action given in \cite[Proposition~10.3]{MR2419901}; denote by $\mathcal{G}$ the groupoid of germs of $(\theta,S,\widehat{E})$; for $e \in E$, denote by $1_e \in C(\widehat{E})$ the characteristic function on $\dom(\theta_e)$; for $s \in S$, denote by $\Theta_s:=\{[s,\phi] : \phi \in \dom(\theta_{s^*s}) \}$ which is an open subset of $\mathcal{G}$, by $\delta_s 1_{s^*s}:=1_{s^*s} \circ s$ on $\Theta_s$ and vanishes outside $\Theta_s$, by $1_{ss^*} \delta_s :=1_{ss^*} \circ r$ on $\Theta_s$ and vanishes outside $\Theta_s$ (notice that $\delta_s 1_{s^*s}, 1_{ss^*} \delta_s$ may not be continuous on $\mathcal{G}$).

\cite[Proposition~7.5]{MR2419901} illustrates the operations of $C^*(\mathcal{G})$. Here are some additional useful facts about $C^*(\mathcal{G}): 1_{1_{C^*(P)}} \delta_{1_{C^*(P)}}=1_{C^*(\mathcal{G})}; 1_{0_{C^*(P)}} \delta_{0_{C^*(P)}}=0_{C^*(\mathcal{G})}$; for any $s,t \in S,\delta_s 1_{s^*s}=1_{ss^*} \delta_s;( 1_{ss^*} \delta_s)(1_{tt^*} \delta_t)=1_{(st)(st)^*}\delta_{st}$; and $\{1_{ss^*} \delta_s:s \in S\}$ generates $C^*(\mathcal{G})$.

\begin{thm}\label{C*(P) is gpoid C*-alg}
$C^*(P) \cong C^*(\mathcal{G})$.
\end{thm}
\begin{proof}
For $p \in P$, define $V_p:=1_{v_p v_p^*}\delta_{v_p}$. For $X \in \mathcal{J}(P)$, define $\epsilon_X:=1_{e_X}\delta_{e_X}$. With the help of \cite[Proposition~7.5]{MR2419901}, we get that $\{V_p\}_{p \in P}$ is a family of isometries and $\{\epsilon_X\}_{X \in \mathcal{J}(P)}$ is a family of projections of $C^*(\mathcal{G})$ satisfying Conditions~(\ref{v_p v_q=v_pq})--(\ref{e_X e_Y}) of Definition~\ref{define semigp C*-alg}. We verify Condition~(\ref{v_P e_X v_p*}) of Definition~\ref{define semigp C*-alg} as an example.

For $p \in P, X \in \mathcal{J}(P)$, we calculate that
\begin{align*}
V_p \epsilon_X V_p^*&=(1_{v_p v_p^*} \delta_{v_p})(1_{e_X} \delta_{e_X})(1_{v_p v_p^*} \delta_{v_p})^*
\\&=(1_{v_p v_p^*} \delta_{v_p})(1_{e_X} \delta_{e_X})(1_{v_p^* v_p} \delta_{v_p^*})
\\&=(1_{(v_p e_X)(v_p e_X)^*}\delta_{v_p e_X})(1_{v_p^* v_p} \delta_{v_p^*})
\\&=1_{(v_pe_Xv_p^*)(v_pe_Xv_p^*)^*}\delta_{v_pe_Xv_p^*}
\\&=1_{e_{pX}}\delta_{e_{pX}}
\\&=\epsilon_{pX}.
\end{align*}
So there exists a homomorphism $\pi:C^*(P) \to C^*(\mathcal{G})$ such that $\pi(v_p)=V_p,\pi(e_X)=\epsilon_X$ for all $p \in P, X \in \mathcal{J}(P)$.

Conversely, we observe that by the definition of $S, S$ is contained in $C^*(P)$. Denote by $\iota:S \to C^*(P)$ the embedding. Take a faithful representation $\rho:C^*(P) \to B(H)$. Denote by $\sigma:S \to C^*(\mathcal{G})$ the $*$-homomorphism given in \cite[Proposition~10.13]{MR2419901} such that $\sigma(s)=1_{ss^*}\delta_s$ for all $s \in S$. By \cite[Theorem~10.14]{MR2419901}, there exists a homomorphism $h:C^*(\mathcal{G}) \to B(H)$ such that $h \circ \sigma=\rho \circ \iota$. For any $s \in S$, we have $h(1_{ss^*} \delta_s)=h(\sigma(s))=\rho(\iota(s)) \in \rho(C^*(P))$. So $h(C^*(\mathcal{G})) \subset \rho(C^*(P))$. Hence we obtain a homomorphism $\rho^{-1}\circ h:C^*(\mathcal{G}) \to C^*(P)$.

For $p \in P$, we have
\[
\rho^{-1} \circ h \circ \pi(v_p)=\rho^{-1} \circ h (1_{v_p v_p^*}\delta_{v_p})=\rho^{-1} \circ h \circ \sigma(v_p)=\rho^{-1} \circ \rho \circ \iota(v_p)=v_p.
\]
For $X \in \mathcal{J}(P)$, we have
\[
\rho^{-1} \circ h \circ \pi(e_X)=\rho^{-1} \circ h (1_{e_X}\delta_{e_X})=\rho^{-1} \circ h \circ \sigma(e_X)=\rho^{-1} \circ \rho \circ \iota(e_X)=e_X.
\]
So $\rho^{-1} \circ h \circ \pi=\id_{C^*(P)}$.

On the other hand, for any $s \in S$, we have
\[\pi \circ \rho^{-1} \circ h(1_{ss^*}\delta_s)=\pi \circ \rho^{-1} \circ h \circ \sigma(s)=\pi \circ \rho^{-1} \circ \rho \circ \iota(s)=\pi(s).
\]
Write $s=v_{p_1}^*v_{q_1}\cdots v_{p_n}^*v_{q_n}$ for some $n \geq 1$, and $p_1,q_1,\dots, p_n,q_n \in P$. We compute that
\begin{align*}
\pi(s)&=(1_{v_{p_1}v_{p_1}^*} \delta_{v_{p_1}})^*(1_{v_{q_1}v_{q_1}^*} \delta_{v_{q_1}})\cdots(1_{v_{p_n}v_{p_n}^*} \delta_{v_{p_n}})^*(1_{v_{q_n}v_{q_n}^*} \delta_{v_{q_n}})
\\&=(1_{v_{p_1}^*v_{p_1}} \delta_{v_{p_1}^*})(1_{v_{q_1}v_{q_1}^*} \delta_{v_{q_1}})\cdots(1_{v_{p_n}^*v_{p_n}} \delta_{v_{p_n}^*})(1_{v_{q_n}v_{q_n}^*} \delta_{v_{q_n}})
\\&=(1_{(v_{p_1}^*v_{q_1})(v_{p_1}^*v_{q_1})^*} \delta_{v_{p_1}^*v_{q_1}})\cdots(1_{(v_{p_n}^*v_{q_n})(v_{p_n}^*v_{q_n})^*} \delta_{v_{p_n}^*v_{q_n}})
\\&=(1_{(v_{p_1}^*v_{q_1}\cdots v_{p_n}^*v_{q_n})(v_{p_1}^*v_{q_1}\cdots v_{p_n}^*v_{q_n})^*} \delta_{v_{p_1}^*v_{q_1}\cdots v_{p_n}^*v_{q_n}})
\\&=1_{ss^*}\delta_s.
\end{align*}
So $\pi \circ \rho^{-1} \circ h=\id_{C^*(\mathcal{G})}$. Hence $C^*(P) \cong C^*(\mathcal{G})$.
\end{proof}

In order to prove our second main theorem, we need a couple of lemmas and more notation.

\begin{lemma}\label{widehat{E}_vee inv}
Let $s \in S$ and let $\phi \in \widehat{E}$ such that for $X,Y \in \mathcal{J}(P)$, if $X \cup Y \in \mathcal{J}(P)$ then $\phi(e_{X \cup Y})=\phi(e_X)+\phi(e_Y)-\phi(e_{X \cap Y})$. Then for $X,Y \in \mathcal{J}(P)$, we have $X \cup Y \in \mathcal{J}(P) \implies \phi(s^*e_{X \cup Y}s)=\phi(s^* e_X s)+\phi(s^*e_Y s)-\phi(s^* e_{X \cap Y}s)$.
\end{lemma}
\begin{proof}
Write $s=v_{p_1}^*v_{q_1}\cdots v_{p_n}^* v_{q_n}$ for some $n \geq 1,p_1,q_1,\dots,p_n,q_n \in P$. Fix $X,Y \in \mathcal{J}(P)$ with $X \cup Y \in \mathcal{J}(P)$. Then we compute that
\begin{align*}
\phi(s^*e_{X \cup Y}s)&=\phi(v_{q_n}^*v_{p_n}\cdots v_{q_1}^* v_{p_1}e_{X \cup Y}v_{p_1}^*v_{q_1}\cdots v_{p_n}^* v_{q_n})
\\&=\phi(e_{q_n^{-1}p_n\cdots q_1^{-1}p_1 (X \cup Y)})
\\&=\phi(e_{q_n^{-1}p_n\cdots q_1^{-1}p_1 X \cup q_n^{-1}p_n\cdots q_1^{-1}p_1 Y})
\\&=\phi(e_{q_n^{-1}p_n\cdots q_1^{-1}p_1 X})+\phi(e_{q_n^{-1}p_n\cdots q_1^{-1}p_1 Y})-\phi(e_{q_n^{-1}p_n\cdots q_1^{-1}p_1 X \cap q_n^{-1}p_n\cdots q_1^{-1}p_1 Y})
\\& (\text{by the hypothesis})
\\&=\phi(e_{q_n^{-1}p_n\cdots q_1^{-1}p_1 X})+\phi(e_{q_n^{-1}p_n\cdots q_1^{-1}p_1 Y})-\phi(e_{q_n^{-1}p_n\cdots q_1^{-1}p_1 (X \cap Y)})
\\& (\text{by the left cancellative of $P$})
\\&=\phi(s^* e_X s)+\phi(s^*e_Y s)-\phi(s^* e_{X \cap Y}s).
\end{align*}
So we are done.
\end{proof}

Denote by $\widehat{E}_{\vee}:=\{\phi \in \widehat{E}:\forall X,Y \in \mathcal{J}(P), X \cup Y \in \mathcal{J}(P) \implies \phi(e_{X \cup Y})=\phi(e_X)+\phi(e_Y)-\phi(e_{X \cap Y})\}$. Then $\widehat{E}_{\vee}$ is a closed invariant subset of $\widehat{E}$ by Lemma~\ref{widehat{E}_vee inv}. Denote by $\mathcal{I}(\widehat{E}_{\vee})$ the set of all bijections between subsets of $\widehat{E}_{\vee}$ and we obtain an action $\theta^{\vee}:S \to \mathcal{I}(\widehat{E}_{\vee})$ such that $\theta^{\vee}_s=\theta_s \vert_{\dom(\theta_s) \cap \widehat{E}_{\vee}}$. Denote by $\mathcal{G}_{\vee}$ the groupoid of germs of $(\theta^{\vee},S,\widehat{E}_{\vee})$. For $e \in E$, denote by $1^\vee_e \in C(\widehat{E}_{\vee})$ the characteristic function on $\dom(\theta^{\vee}_e)$; for $s \in S$, denote by $\Theta^{\vee}_s:=\{[s,\phi] : \phi \in \dom(\theta^{\vee}_{s^*s}) \}$ which is an open subset of $\mathcal{G}_{\vee}$, by $\delta_s 1^\vee_{s^*s}:=1^\vee_{s^*s} \circ s$ on $\Theta^{\vee}_s$ and vanishes outside $\Theta^{\vee}_s$, by $1^\vee_{ss^*} \delta_s :=1^\vee_{ss^*} \circ r$ on $\Theta^{\vee}_s$ and vanishes outside $\Theta^{\vee}_s$. We have $1^\vee_{1_{C^*(P)}} \delta_{1_{C^*(P)}}=1_{C^*(\mathcal{G}_{\vee})}; 1^\vee_{0_{C^*(P)}} \delta_{0_{C^*(P)}}=0_{C^*(\mathcal{G}_{\vee})}$; for any $s,t \in S,\delta_s 1^\vee_{s^*s}=1^\vee_{ss^*} \delta_s;( 1^\vee_{ss^*} \delta_s)(1^\vee_{tt^*} \delta_t)=1^\vee_{(st)(st)^*}\delta_{st}$; and $\{1^\vee_{ss^*} \delta_s:s \in S\}$ generates $C^*(\mathcal{G}_{\vee})$.

\begin{lemma}\label{characterization of supp in widehat{E}_vee}
Let $\sigma$ be a representation of $S$ on a Hilbert space $H$. Then $\sigma$ is supported in $\widehat{E}_\vee$ in the sense of \cite[Definition~10.11]{MR2419901} if and only if $\sigma(e_{X \cup Y})=\sigma(e_X)+\sigma(e_Y)-\sigma(e_{X \cap Y})$ for all $X,Y \in \mathcal{J}(P)$ with $X \cup Y\in \mathcal{J}(P)$.
\end{lemma}
\begin{proof}
Denote by $\pi_\sigma$ the representation of $C(\widehat{E})$ induced from $\sigma$ from \cite[Proposition~10.6]{MR2419901}. Suppose that $\sigma$ is supported in $\widehat{E}_\vee$. Fix $X,Y \in \mathcal{J}(P)$ with $X \cup Y\in \mathcal{J}(P)$. By the definition of $\widehat{E}_\vee$, we have $1_{e_{X \cup Y}}+1_{e_{X \cap Y}}-1_{e_X}-1_{e_Y} \in C_0(\widehat{E} \setminus \widehat{E}_\vee)$. Since $\sigma$ is supported in $\widehat{E}_\vee$, we obtain $\pi_\sigma(1_{e_{X \cup Y}}+1_{e_{X \cap Y}}-1_{e_X}-1_{e_Y})=0$. So $\sigma(e_{X \cup Y})=\sigma(e_X)+\sigma(e_Y)-\sigma(e_{X \cap Y})$.

Conversely, suppose that $\sigma(e_{X \cup Y})=\sigma(e_X)+\sigma(e_Y)-\sigma(e_{X \cap Y})$ for all $X,Y \in \mathcal{J}(P)$ with $X \cup Y\in \mathcal{J}(P)$. Then $1_{e_{X \cup Y}}+1_{e_{X \cap Y}}-1_{e_X}-1_{e_Y} \in \ker(\pi_\sigma)$ for all $X,Y \in \mathcal{J}(P)$ with $X \cup Y\in \mathcal{J}(P)$. Notice that $1_{e_{X \cup Y}}+1_{e_{X \cap Y}}-1_{e_X}-1_{e_Y} \in C_0(\widehat{E} \setminus \widehat{E}_\vee)$ for all $X,Y \in \mathcal{J}(P)$ with $X \cup Y\in \mathcal{J}(P)$. So in order to show that $\sigma$ is supported in $\widehat{E}_\vee$, we invoke the Stone-Weierstrass theorem to prove that $A:=C^*(1_{e_{X \cup Y}}+1_{e_{X \cap Y}}-1_{e_X}-1_{e_Y} :X,Y ,X \cup Y\in \mathcal{J}(P))=C_0(\widehat{E} \setminus \widehat{E}_\vee)$. Fix $\phi \in \widehat{E} \setminus \widehat{E}_\vee$. Then there exist $X,Y \in \mathcal{J}(P)$ with $X \cup Y \in \mathcal{J}(P)$ such that $\phi(e_{X \cup Y})+\phi(e_{X \cap Y}) \neq \phi(e_X)+\phi(e_Y)$. So $(1_{e_{X \cup Y}}+1_{e_{X \cap Y}}-1_{e_X}-1_{e_Y})(\phi) \neq 0$. Fix $\phi_1\neq \phi_2 \in \widehat{E} \setminus \widehat{E}_\vee$. Then there exists $Z \in \mathcal{J}(P)$ such that $\phi_1(e_Z) \neq \phi_2(e_Z)$. Without loss of generality, we may assume that $\phi_1(e_Z)=1$ and $\phi_2(e_Z)=0$. Since $\phi_1 \notin \widehat{E}_\vee$, there exist $X,Y \in \mathcal{J}(P)$ with $X \cup Y \in \mathcal{J}(P)$ such that $\phi_1(e_{X \cup Y})+\phi_1(e_{X \cap Y}) \neq \phi_1(e_X)+\phi_1(e_Y)$. Observe that $Z \cap X,Z \cap Y,(Z \cap X) \cup (Z \cap Y) \in \mathcal{J}(P)$. So $(1_{e_{(Z \cap X) \cup (Z \cap Y)}}+1_{e_{(Z \cap X) \cap (Z \cap Y)}}-1_{e_{Z \cap X}}-1_{e_{Z \cap Y}})(\phi_1)=1$ and $(1_{e_{(Z \cap X) \cup (Z \cap Y)}}+1_{e_{(Z \cap X) \cap (Z \cap Y)}}-1_{e_{Z \cap X}}-1_{e_{Z \cap Y}})(\phi_2)=0$. Therefore we are done.
\end{proof}

\begin{thm}
$C^{*(\vee)}(P) \cong C^*(\mathcal{G}_\vee)$.
\end{thm}
\begin{proof}
For $p \in P$, define $W_p:=1^\vee_{v_p v_p^*}\delta_{v_p}$. For $X \in \mathcal{J}(P)$, define $F_X:=1^\vee_{e_X}\delta_{e_X}$. With almost the same proof as Theorem~\ref{C*(P) is gpoid C*-alg}, $\{W_p\}_{p \in P}$ is a family of isometries and $\{F_X\}_{X \in \mathcal{J}(P)}$ is a family of projections of $C^*(\mathcal{G}_{\vee})$ satisfying Conditions~(\ref{v_p v_q=v_pq})--(\ref{e_X e_Y}) of Definition~\ref{define semigp C*-alg}. We verify Condition~(\ref{e_X cup Y}) of Definition~\ref{define semigp C*-alg}. Fix $X, Y \in \mathcal{J}(P)$ with $X \cup Y \in \mathcal{J}(P)$ and fix $[s,\phi] \in \mathcal{G}_\vee$. If $[s,\phi] \notin \Theta^\vee_{e_{X \cup Y}}$, then by \cite[Proposition~7.6]{MR2419901}, $[s,\phi] \notin \Theta^\vee_{e_{X}},\Theta^\vee_{e_{Y}},\Theta^\vee_{e_{X \cap Y}}$. It follows that $F_{X \cup Y}([s,\phi])=(F_X+F_Y-F_{X \cap Y})([s,\phi])=0$. So we may assume that $[s,\phi] \in \Theta^\vee_{e_{X \cup Y}}$. Then $\phi(e_{X \cup Y})=1$ and there exists $Z \in \mathcal{J}(P)$ such that $\phi(e_Z)=1$ and $e_{X \cup Y}e_Z=s e_Z$. Since $\phi \in \widehat{E}_\vee,\phi(e_X)+\phi(e_Y)-\phi(e_{X \cap Y})=1$. So $\phi(e_X)$ and $\phi(e_Y)$ can not be zero at the same time.

Case 1. $\phi(e_X)=\phi(e_Y)=1$. Then $\phi(e_{X \cap Y})=1$. Notice that $e_W (e_Z e_W)=s(e_Z e_W)$ for all $\mathcal{J}(P) \ni W \subset X \cup Y$. So $[s,\phi] \in \Theta^\vee_{e_{X}},\Theta^\vee_{e_{Y}},\Theta^\vee_{e_{X \cap Y}}$. Hence $F_{X \cup Y}([s,\phi])=(F_X+F_Y-F_{X \cap Y})([s,\phi])=1$.

Case 2. $\phi(e_X)=1,\phi(e_Y)=0$. Then $\phi(e_{X \cap Y})=0$. So $[s,\phi] \in \Theta^\vee_{e_{X}},[s,\phi] \notin \Theta^\vee_{e_{Y}},\Theta^\vee_{e_{X \cap Y}}$. Hence $F_{X \cup Y}([s,\phi])=(F_X+F_Y-F_{X \cap Y})([s,\phi])=1$.

Case 3. $\phi(e_X)=0,\phi(e_Y)=1$. Then $\phi(e_{X \cap Y})=0$. So $[s,\phi] \in \Theta^\vee_{e_{Y}},[s,\phi] \notin \Theta^\vee_{e_{X}},\Theta^\vee_{e_{X \cap Y}}$. Hence $F_{X \cup Y}([s,\phi])=(F_X+F_Y-F_{X \cap Y})([s,\phi])=1$.

We conclude that $F_{X \cup Y}=F_X+F_Y-F_{X \cap Y}$. So there exists a homomorphism $\pi:C^{*(\vee)}(P) \to C^*(\mathcal{G}_\vee)$ such that $\pi(w_p)=W_p,\pi(f_X)=F_X$ for all $p \in P, X \in \mathcal{J}(P)$.

Conversely, denote by $\iota:S \to C^*(P)$ the embedding, denote by $q:C^*(P) \to C^{*(\vee)}(P)$ the quotient map, and take a faithful representation $\rho:C^{*(\vee)}(P) \to B(H)$. Then $\sigma:=\rho \circ q \circ \iota$ is a representation of $S$ on $H$. By Condition~(\ref{e_X cup Y}) of Definition~\ref{define semigp C*-alg}, $\sigma(e_{X \cup Y})=\sigma(e_X)+\sigma(e_Y)-\sigma(e_{X \cap Y})$ for all $X,Y \in \mathcal{J}(P)$ with $X \cup Y\in \mathcal{J}(P)$. By Lemma~\ref{characterization of supp in widehat{E}_vee}, $\sigma$ is supported in $\widehat{E}_\vee$. Denote by $\sigma^\vee:S \to C^*(\mathcal{G}_\vee)$ the $*$-homomorphism given in \cite[Proposition~10.13]{MR2419901} such that $\sigma^\vee(s)=1^\vee_{ss^*}\delta_s$ for all $s \in S$. By \cite[Theorem~10.14]{MR2419901}, there exists a homomorphism $h:C^*(\mathcal{G}_\vee) \to B(H)$ such that $h\circ \sigma^\vee=\sigma$.

Similarly to Theorem~\ref{C*(P) is gpoid C*-alg}, we can prove that $\pi\circ\rho^{-1}\circ h=\id_{C^*(\mathcal{G}_\vee)}$ and $\rho^{-1}\circ h\circ \pi=C^{*(\vee)}(P)$. Hence $C^{*(\vee)}(P) \cong C^*(\mathcal{G}_\vee)$.
\end{proof}

\end{document}